\documentclass[11pt,reqno]{amsart}
\usepackage{amsmath, amsthm, amssymb}
\usepackage{mathrsfs}
\usepackage{array}
\usepackage{caption}
\evensidemargin \oddsidemargin
\newtheorem{thm}{Theorem}

\newtheorem{lem}{Lemma}

\newtheorem{prop}{Proposition}

\numberwithin{equation}{section} \numberwithin{thm}{section}
\numberwithin{lem}{section} \numberwithin{problem}{section}
\numberwithin{cor}{section}
\def\grm{{\mathfrak m}}\def\grM{{\mathfrak M}}\def\grN{{\mathfrak N}}\def\grn{{\mathfrak n}}\def\grp{{\mathfrak p}}\def\grP{{\mathfrak P}}
\newcommand{\mmod}[1]{\,\,(\text{\rm mod}\,\, #1)}

\def\grm{{\mathfrak m}}\def\grM{{\mathfrak M}}\def\grN{{\mathfrak N}}\def\grn{{\mathfrak n}}

\begin{document}
\title[Sums of three cubes]{On Waring's problem in sums of three cubes for smaller powers}
\author[Javier Pliego]{Javier Pliego}
\address{Purdue Mathematical Science Building, 150 N University St, West Lafayette, IN 47907, United States of America}

\email{jpliegog@purdue.edu}

\begin{abstract} We give an upper bound for the minimum $s$ with the property that every sufficiently large integer can be represented as the sum of $s$ positive $k$-th powers of integers represented as the sum of three positive cubes for the cases $2\leq k\leq 4.$
\end{abstract}
\maketitle

\section{Introduction}
Additive problems involving small powers of positive integers have led to a vast development of new ideas and techniques in the application of the Hardy-Littlewood method which often cannot be extended to the setting of general $k$-th powers. Finding the least number $s$ such that for every sufficiently large integer $n$ then
\begin{equation}\label{war}n=x_{1}^{k}+\ldots+x_{s}^{k},\end{equation}
where $x_{i}\in\mathbb{N},$ might be among the most studied examples. We denote such number $s$ by $G(k)$. Let $\mathscr{C}$ be the set of integers represented as a sum of three positive integral cubes. In this memoir we shall be concerned with the function $G_{3}(k)$, which we define as the minimum $s$ such that (\ref{war}) is soluble with $x_{i}\in\mathscr{C}$, for the cases $2\leq k\leq 4.$

Providing the precise value of $G(k)$ is still an open question for most $k$, the cases $k=2,4$ being precisely the only ones solved. Lagrange showed in 1770 that $G(2)=4$ and Davenport \cite{Dav} proved in 1939 the identity $G(4)=16,$ and though it is believed that $G(3)=4,$ the best current upper bound is $G(3)\leq 7$ due to Linnik \cite{Lin}. 

Not very much is known about $\mathscr{C}$. In fact, it isn't even known whether it has positive density or not, the best current lower bound on the cardinality of the set being $$\mathcal{N}(X)=\lvert\mathscr{C}\cap [1,X]\rvert\gg X^{\beta},$$ where  
$\beta=0.91709477,$ due to Wooley \cite{Woo3}. We note that under some unproved assumptions on the zeros of some Hasse-Weil $L$-functions, Hooley (\cite{Hol1}, \cite{Hol2}) and Heath-Brown \cite{Hea} showed using different procedures that
$$\sum_{n\leq X}r_{3}(n)^{2}\ll X^{1+\varepsilon},$$ where $r_{3}(n)$ is the number of representations of $n$ as sums of three positive integral cubes, which implies by applying a standard Cauchy-Schwarz argument that $\mathcal{N}(X)\gg X^{1-\varepsilon}.$
This lack of understanding of the cardinality of the set also prevents us from understanding its distribution over arithmetic progressions, which often comes into play on the major arc analysis. Therefore, even if the exponents $k=2,4$ are well understood for the original problem, it turns out to be much harder when we restrict the variables to lie on $\mathscr{C}.$ In this paper we establish the following bounds for $G_{3}(k)$.
\begin{thm}\label{thm1.1}
One has $G_{3}(2)\leq 8,$ $G_{3}(3)\leq 17$ and $G_{3}(4)\leq 57.$
\end{thm}
We are far from knowing whether these estimates are good or bad, since the only lower bounds that we have for the above quantities are $4\leq G(3)\leq G_{3}(3)$ and $16=G(4)\leq G_{3}(4).$ For the case $k=2$ though we can actually do better. We take, for convenience, an integer $j\geq 0$, and observe that the only solution to \begin{equation}\label{este}x_{1}^{2}+x_{2}^{2}+x_{3}^{2}+x_{4}^{2}=2^{6+12j}\end{equation} with $x_{i}\in\mathbb{N}$ is $x_{1}=x_{2}=x_{3}=x_{4}=2^{2+6j}.$ This can be seen by taking the equation (\ref{este}) modulo $8$, realising that one must have $2\mid x_{i}$ for every $i$ and iterating the process. However, one has $2^{2+6j}\equiv 4 \mmod{9}$, and no number congruent to $4\mmod{9}$ can be written as the sum of three cubes. Therefore, there are infinitely many numbers for which (\ref{este}) doesn't have any solution with $x_{i}\in\mathscr{C}.$ The preceding remark implies then the bound $5\leq G_{3}(2).$

Our proof of Theorem \ref{thm1.1} is based on the application of the Hardy-Littlewood method. In the setting of this paper, the constraint which prevents us from taking fewer variables is the treatment of the minor arcs discussed in Section \ref{sec2}. In order to analyse them we utilise an argument of Vaughan \cite[Lemma 3.4]{Vau3} to bound certain families of exponential sums over the minor arcs together with non-optimal estimates of sums of the shape \begin{equation*}\sum_{x\leq X}a_{x}^{2},  \ \ \text{where} \  a_{x}={\rm card}\big\{\mathbf{x}\in\mathbb{N}^{3}:\ x=x_{1}^{3}+x_{2}^{3}+x_{3}^{3},\ \ x_{2},x_{3}\in \mathcal{A}(P,P^{\eta})\big\}\end{equation*} with $\eta>0$ being a small enough parameter and
$$\mathcal{A}(Y,R)=\{n\in [1,Y]\cap \mathbb{N}: p\mid n\text{ and $p$ prime}\Rightarrow p\leq R\}.$$ Here, the reader may find it useful to observe that it is a consequence of Vaughan and Montgomery \cite[Theorem 7.2]{Mon} that \begin{equation*}\label{smo}\text{card}\big(\mathcal{A}(P,P^{\eta})\big)=c_{\eta}P+O\big(P/\log P\big)\end{equation*} for some constant $c_{\eta}>0$ that only depends on $\eta$. 
In order to briefly discuss the outcome that follows after applying the argument of Vaughan we introduce the exponential sum \begin{equation}\label{Weig}W(\alpha)=\sum_{M/2\leq p\leq M}\sum_{H/2\leq h\leq H}b_{h}e(\alpha p^{3k}h^{k}),\end{equation} where $M,H>0$ and $(b_{h})_{h}$ are weights which the reader should think of being  $(a_{h})_{h}$ and $p$ runs over prime numbers. In order to make the argument work, the parameters $M$ and $H$ must be subjected to the constraint $\max(M^{5-1/k}, M^{2^{k-1}})\leq H$. The saving over the natural bound $HM$ for $W(\alpha)$ obtained with the method is roughly speaking of size $M^{1/2}H^{-1/24}$, which makes the estimate obtained worse than trivial for $k\geq 5.$ 

Let $\tilde{G}_{3}(k)$ be the minimum number such that for $s\geq \tilde{G}_{3}(k)$, the anticipated asymptotic formula for the number of representations in the shape (\ref{war}) with $x_{i}\in\mathscr{C}$ counted with multiplicities holds. Similarly, let $\tilde{G}(k)$ be the minimum number such that for $s\geq \tilde{G}(k)$, the anticipated asymptotic formula in the classical Waring's problem holds. In a forthcoming paper \cite{Pli}, the author shows that $\tilde{G}_{3}(k)\leq 9k^{2}-k+2$ and $G_{3}(k)\leq 3k\big(\log k+\log\log k+O(1)\big),$ 
which the reader may compare to $\tilde{G}(k)\leq k^{2}-k+O(\sqrt{k})$ due to Bourgain \cite{Bou} and $G(k)\leq k\log k+k\log\log k+O(k)$ due to Wooley \cite{Woo7}. A weakness of the underlying methods employed in that paper is that in most of the sums of three cubes employed in the representation, all but one of the cubes is fixed in the associated analysis.

A naive approach to bounding $G_{3}(k)$ would then be to replace each sum of three cubes by the specialisation $3x^{3}$, and this suggests a bound of the shape $G_{3}(k)\leq G(3k)$. With this idea in mind, the bounds $G(6)\leq 24$ due to Vaughan and Wooley \cite{VauWoo1}, $G(9)\leq 47$ and $G(12)\leq 72$ due to Wooley \cite{Woo6} reveal that the methods used in this memoir improve what would have been the trivial approach and confirms that we are actually using the three integral cubes non-trivially in our argument. For the cases $k=2,3$ we combine the pointwise bound obtained for $W(\alpha)$ over the minor arcs with some restriction estimates involving the coefficients $(a_{m})_{m}$. When $k=4$ we instead use a bound for a mean value of smooth Weyl sums of exponent $12$. The estimate for $W(\alpha)$ obtained here is then robust enough to enable us to gain $15$ variables from the trivial approach over the minor arcs and allows us to prune back to a narrower set of major arcs.

The paper is organized as follows. In Section \ref{sec2} we use Vaughan methods to estimate $W(\alpha)$ and provide bounds for the contribution of the minor arcs which are good enough for our purposes when $k=2,3$. We approximate the generating functions of the problem on a narrower set of major arcs in Section \ref{sec3}. In Sections \ref{sec4}, \ref{sec5} and \ref{sec6} we only consider the exponents $k=2,3$, whereas in Section \ref{sec7} we prove the theorem for $k=4.$ Sections \ref{sec4} and \ref{sec5} are devoted to the study of the singular series and the singular integral respectively. We then prune back to the narrower set of arcs to show a lower bound for the major arc contribution in Section \ref{sec6}.

Unless specified, any lower case letter $\mathbf{x}$ written in bold will denote a triple of integers $(x_{1},x_{2},x_{3})$. We will write $a\leq \mathbf{V}\leq b$ when $a\leq v_{i}\leq b$ for $1\leq i\leq n$. As usual in analytic number theory, for each $x\in\mathbb{R}$ we denote $\exp(2\pi i x)$ by $e(x)$, and for each natural number $q$ then $e(x/q)$ will be written as $e_{q}(x).$ We use $\ll$ and $\gg$ to denote Vinogradov's notation, and write $A\asymp B$ whenever $A\ll B\ll A$. When $\varepsilon$ appears in any bound, it will mean that the bound holds for every $\varepsilon>0$, though the implicit constant then may depend on $\varepsilon$. We adopt the convention that whenever we write $\delta$ in the computations we mean that there exists a positive constant $\delta$ for which the bound holds. We write $p^{r}|| n$ to denote that $p^{r}| n$ but $p^{r+1}\nmid n.$

\emph{Acknowledgements}: The author's work was supported in part by a European Research Council Advanced
Grant under the European Union’s Horizon 2020 research and innovation programme via grant agreement No. 695223 during his studies at the University of Bristol. It was completed while the author was visiting Purdue University under Trevor Wooley's supervision. The author would like to thank him for his guidance and helpful comments, an anonymous referee for useful remarks and both the University of Bristol and Purdue University for their support and hospitality.
\section{Minor arc estimates}\label{sec2}
As mentioned in the introduction, we provide an estimate for the exponential sum $W(\alpha)$ by using methods of Vaughan. We make use of a Hardy-Littlewood dissection and combine both the bound for $W(\alpha)$ and a restriction estimate of a certain mean value to bound the minor arc contribution for the cases $k=2,3$. We also remark that the estimate for $W(\alpha)$ is also used in Sections \ref{sec6} and \ref{sec7} to prune the major arcs back to a narrower set of arcs. Before going into the proof of the main lemma, it is convenient to write
\begin{equation}\label{Saq}S_{k}(q,a)=\sum_{r=1}^{q}e_{q}(ar^{k}).\end{equation} We also introduce the multiplicative function $\tau_{k}(q)$ by defining $\tau_{2}(q)=q^{-1/2}$ and $$\ \ \ \ \ \ \ \ \ \ \ \ \ \  \tau_{k}(p^{uk+v})=p^{-u-1}\ \ \text{when $u\geq 0$ and $2\leq v\leq k$}$$ and $$\tau_{k}(p^{uk+1})=kp^{-u-1/2}\ \ \text{ when $u\geq 0$}$$for $k=3,4.$ Observe that with this definition then one has the bound \begin{equation}\label{tauk}\tau_{k}(q)\ll q^{-1/k},\end{equation} and the proof of Theorem 4.2 of \cite{Vau} yields
\begin{equation}\label{skk}q^{-1}S_{k}(q,a)\ll \tau_{k}(q).\end{equation}
\begin{lem}\label{lema1}
Let $H,M>0$ such that $\max(M^{5-1/k},M^{2^{k-1}})\leq H.$ Let $\alpha\in [0,1)$. Suppose that $\alpha=a/q+\beta$, where $a\in\mathbb{Z}$ and $q\in\mathbb{N}$ with $(a,q)=1$ and such that $q\leq Y,$ and $\lvert \beta\rvert\leq q^{-1}Y^{-1},$ where $Y$ is a parameter in the range $M^{k}\leq Y\leq H^{k}M^{2k}.$ Then the exponential sum $W(\alpha)$, defined in (\ref{Weig}), satisfies
\begin{equation}\label{boundW}W(\alpha)\ll H^{\varepsilon}\Bigg(HM+\frac{\tau_{k}(q)HM^{2}}{1+M^{3k}H^{k}\lvert \alpha-a/q\rvert}\Bigg)^{1/2}\Bigg(\sum_{H/2\leq h\leq H}\lvert b_{h}\rvert^{2}\Bigg)^{1/2}.\end{equation} 
\end{lem}

\begin{proof}
For the sake of simplicity we will not write the limits of summation for $p$ and $h$ throughout the rest of the section. We apply Cauchy-Schwarz to obtain
\begin{align}\label{2.2}W(\alpha)&\ll \Big(\sum_{h}\lvert b_{h}\rvert^{2}\Big)^{1/2}\Big(\sum_{h}\sum_{ p_{1},p_{2}}e\big(\alpha(p_{1}^{3k}-p_{2}^{3k})h^{k}\big)\Big)^{1/2}\nonumber
\\
&\ll\Big(\sum_{h}\lvert b_{h}\rvert^{2}\Big)^{1/2}\Big(HM+E(\alpha)\Big)^{1/2},
\end{align}
where the term $HM$ comes from the diagonal contribution and $$E(\alpha)=\sum_{h}\sum_{\substack{p_{2}< p_{1}} }e\big(\alpha(p_{1}^{3k}-p_{2}^{3k})h^{k}\big).$$
In order to estimate $E(\alpha)$ we will follow closely the argument of Vaughan \cite[Lemma 3.4]{Vau3}. For a given pair of primes $(p_{1},p_{2})$ we choose $b,r\in\mathbb{N}$ with $(b,r)=1$ and such that $r\leq 2kH^{k-1}$ and $\lvert\alpha(p_{1}^{3k}-p_{2}^{3k})-b/r\rvert\leq (2k)^{-1} r^{-1}H^{1-k}$. Then if $r>H$, an application of Weyl's inequality \cite[Lemma 2.4]{Vau} yields the bound
$$\sum_{h}e\big(\alpha(p_{1}^{3k}-p_{2}^{3k})h^{k}\big)\ll H^{1-2^{1-k}+\varepsilon}\ll H^{1+\varepsilon}M^{-1},$$ where we used the restriction on $M$ at the beginning of the lemma. If on the other hand $r\leq H$ we combine Lemmata 6.1 and 6.2 of \cite{Vau} with (\ref{skk}) to obtain
$$\sum_{h}e\big(\alpha(p_{1}^{3k}-p_{2}^{3k})h^{k}\big)\ll \frac{\tau_{k}(r)H}{1+H^{k}\big\lvert \alpha(p_{1}^{3k}-p_{2}^{3k})-b/r\big\rvert}+r^{1/2+\varepsilon}.$$Consequently, one has that
$$E(\alpha)\ll E_{0}+H^{1+\varepsilon}M+\sum_{(p_{1},p_{2})}H^{1/2+\varepsilon}\ll E_{0}+H^{1+\varepsilon}M,$$ where
$$E_{0}=\sum_{(p_{1},p_{2})\in\mathcal{A}}\frac{\tau_{k}(r)H}{1+H^{k}\big\lvert \alpha(p_{1}^{3k}-p_{2}^{3k})-b/r\big\rvert}$$ and $\mathcal{A}$ is the set of pairs $(p_{1},p_{2})$ with $p_{2}<p_{1}$ for which $r<(6k)^{-1}M^{k}$ and such that $\big\lvert \alpha(p_{1}^{3k}-p_{2}^{3k})-b/r\big\rvert< 2^{-1}r^{-1/k}MH^{-k}$. Note that the contribution of the pairs for which one of the previous two restrictions doesn't hold is $O(HM).$ For each pair $(p_{1},p_{2})$, define $n=p_{2},$ $l=p_{1}-p_{2}$ and $D=\big((n+l)^{3k}-n^{3k}\big)/l$. Then one finds that
\begin{equation}\label{e0}E_{0}\ll \sum_{(n,l)}\frac{\tau_{k}(r)H}{1+H^{k}\big\lvert \alpha lD-b/r\big\rvert},\end{equation}where $(n,l)$ runs over pairs with $1\leq l\leq M$ and $M/2\leq n\leq M$ such that $(n+l,n)=1$ and satisfying the existing bounds on $r$ and $\big\lvert \alpha lD-b/r\big\rvert$.

We next choose for convenience $c,s\in\mathbb{N}$ satisfying $(c,s)=1$ and with the property that $s\leq H^{k}M^{-k}$ and $\lvert \alpha l-c/s\rvert\leq s^{-1}M^{k}H^{-k}$. By the constraint imposed on $M$ and $H$ at the beginning of the lemma we obtain
$$\Big\lvert \frac{c}{s}-\frac{b}{rD}\Big\rvert sDr< DrM^{k}H^{-k}+\frac{1}{2}sr^{1-1/k}MH^{-k}< \frac{3k}{6k}M^{5k-1}H^{-k}+\frac{1}{2}sM^{k}H^{-k}\leq 1.$$ Therefore, one has $crD=bs$, and hence the coprimality condition on $r$ and $b$ yields $r|s.$ Let $s_{0}= s/r$. We then have that $s_{0}\mid D$, whence
$$E_{0}\ll \sum_{s_{0}\mid s}\tau_{k}\Big(\frac{s}{s_{0}}\Big)\sum_{(n,l)}\frac{H}{1+H^{k}D\big\lvert \alpha l-c/s\big\rvert},$$where the sum on $(n,l)$ runs over the same range described after (\ref{e0}) with the conditions $(n+l,n)=1$ and $\big((n+l)^{3k}-n^{3k}\big)/l\equiv 0\pmod{s_{0}}$. Once we fix $l$ then using the above constraints one has that the number of such $n$ is bounded above by $O\big((M/s_{0}+1)s_{0}^{\varepsilon}\big).$ Consequently, we obtain that $E(\alpha)\ll H^{1+\varepsilon}M+H^{\varepsilon}ME_{1},$ where $$E_{1}=\sum_{l\in\mathcal{L}}\frac{\tau_{k}(s)H}{1+H^{k}M^{3k-1}\big\lvert \alpha l-c/s\big\rvert},$$and $\mathcal{L}$ is the set of integers $l\leq M$ for which $s<M^{k}/2$ and $\lvert \alpha l-c/s\rvert <M^{2-3k}H^{-k}$. Now we choose $d,t$ with $(d,t)=1$ satisfying $t\leq M^{k+1}$ and $\big\lvert \alpha-d/t\big\rvert\leq t^{-1}M^{-k-1}.$ One finds that 
$$\Big\lvert \frac{c}{ls}-\frac{d}{t}\Big\rvert slt< stM^{2-3k}H^{-k}+slM^{-k-1}< \frac{1}{2}M^{3-k}H^{-k}+\frac{1}{2}\leq 1.$$
Therefore, one has $ct=dsl$ and hence $s|t$. Let $t_{0}=t/s$. Then it follows that $t_{0}\mid l$, and on defining $l_{0}=l/t_{0}$ we obtain
$$E_{1}\ll \sum_{t_{0}\mid t}\tau_{k}\Big(\frac{t}{t_{0}}\Big)\sum_{l_{0}\leq M/t_{0}}\frac{H}{1+H^{k}M^{3k-1}l_{0}t_{0}\big\lvert \alpha -d/t\big\rvert}\ll \frac{\tau_{k}(t)HM^{1+\varepsilon}}{1+H^{k}M^{3k}\big\lvert \alpha-d/t\big\rvert}.$$
If either $t\geq M^{k}/2$ or $\lvert \alpha -d/t\rvert\geq 2^{-1}t^{-1/k}H^{-k}M^{1-3k}$ then we get $E_{1}\ll HM^{\varepsilon}$ and we would be done. For the remaining cases one finds that
$$\Big\lvert \frac{a}{q}-\frac{d}{t}\Big\rvert qt< \frac{1}{2}qH^{-k}M^{1-3k}t^{1-1/k}+tY^{-1}< \frac{1}{2}YH^{-k}M^{-2k}+\frac{1}{2}M^{k}Y^{-1}\leq 1,$$which implies that $a=d$ and $q=t$, and yields the bound
$$E(\alpha)\ll H^{1+\varepsilon}M+\frac{\tau_{k}(q)H^{1+\varepsilon}M^{2}}{1+H^{k}M^{3k}\big\lvert \alpha-a/q\big\rvert}.$$The combination of this estimate and (\ref{2.2}) proves the lemma.
\end{proof}
Before describing the application of this lemma in the minor arc treatment it is convenient to introduce some notation. Let $n$ be a natural number and take $P= n^{1/3k}$. Define the parameters \begin{equation}\label{MH}\gamma(k)=\displaystyle\frac{3}{3+\max(5-1/k,2^{k-1})},\ \ \ \ \ \ M=P^{\gamma(k)},\ \ \ \ \ \ H=\max(M^{5-1/k},M^{2^{k-1}}).\end{equation}
Note that these choices for $M$ and $H$ maximize the saving obtained for $W(\alpha)$ over the trivial bound in the previous lemma. Take $$H_{1}=\Big(\frac{1}{2}\Big)^{1/3}H^{1/3},\ \ \ \ \ \ H_{2}=\Big(\frac{2}{3}\Big)^{1/3}H^{1/3},\ \ \ \ \ \ H_{3}=\Big(\frac{1}{6}\Big)^{1/3}H^{1/3}.$$
For every triple $\mathbf{x}\in\mathbb{R}^{3}$, consider the function $T(\mathbf{x})=x_{1}^{3}+x_{2}^{3}+x_{3}^{3}.$
Define the sets $$\mathcal{H}=\Big\{(y,\mathbf{y})\in\mathbb{N}^{3}:\ \ \frac{P}{2}\leq y\leq P,\ \ \ \ \ \mathbf{y}\in\mathcal{A}(P,P^{\eta})^{2} \Big\},$$
$$\mathcal{W}=\Big\{(y,\mathbf{y})\in\mathbb{N}^{3}:\ \ \ H_{1}\leq y\leq H_{2},\ \ \ \ \mathbf{y}\in\mathcal{A}(H_{3},P^{\eta})^{2}\Big\},$$ and the corresponding weights
$$a_{x}=\lvert \{\mathbf{x}\in\mathcal{H}:\ x=T(\mathbf{x})\}\rvert,\ \ \ \ \ \ b_{h}=\lvert \{\mathbf{x}\in\mathcal{W}:\ h=T(\mathbf{x})\}\rvert,$$
where $(b_{h})_{h}$ is the choice that we make for the weights of $W(\alpha)$ in (\ref{Weig}). We use $(a_{x})_{x}$ to define the weighted exponential sum
$$h(\alpha)=\sum_{x\leq 3P^{3}}a_{x}e(\alpha x^{k}).$$

Before describing how $h(\alpha)$ and $W(\alpha)$ play a role in the argument we first show upper bounds on the $L^{2}$-norms of the weights which will be used to estimate the minor arc contribution. Let $X>0$, consider $$f(\alpha;X)=\sum_{x\leq X}e(\alpha x^{3}),\ \ \ \ \ \ \ \ f(\alpha;X;X^{\eta})=\sum_{x\in\mathcal{A}(X,X^{\eta})}e(\alpha x^{3})$$ and define the mean value
$$U(X)=\int_{0}^{1}\lvert f(\alpha;X)\rvert^{2}\lvert  f(\alpha;X;X^{\eta})\rvert^{4}d\alpha.$$ It is a consequence of Wooley  \cite[Theorem 1.2]{Woo3} that $U(X)\ll X^{3+1/4-\tau},$ where $\tau=0.00128432.$ Consequently, on considering the underlying diophantine equations due to orthogonality, it follows that
\begin{equation}\label{Up}\sum_{x\leq 3P^{3}}a_{x}^{2}\leq U(P)\ll P^{3+1/4-\tau},\ \ \ \ \ \ \ \ \sum_{H/2\leq h\leq H} b_{h}^{2}\leq U(H^{1/3})\ll H^{13/12-\tau/3}.\end{equation}
The reader may note that we didn't write the entire decimal expression of $\tau$, so the bound for $U(X)$ holds for a slightly bigger $\tau$. Therefore, whenever we encounter bounds with the mean value $U(X)$ involved, we can omit the parameter $\varepsilon$ in the exponents.

Take $s(k)=2^{k}$ when $k=2,3$ and define $t(k)$ by $t(2)=4$ and $t(3)=9$. For ease of notation we will just write $s$ and $t$ instead of $s(k)$ and $t(k)$ throughout the paper.  Let $R(n)$ be the number of solutions of the equation
$$n=\sum_{i=1}^{t}T(p_{i}\mathbf{x}_{i})^{k}+\sum_{i=t+1}^{s+t}T(\mathbf{x}_{i})^{k},$$ where $\mathbf{x}_{i}\in \mathcal{W}$ for $1\leq i\leq t$ with $M/2\leq p_{i}\leq M$ prime and $\mathbf{x}_{i}\in \mathcal{H}$ for $t+1\leq i\leq s+t$. Note that by orthogonality then $$R(n)=\int_{0}^{1}h(\alpha)^{s}W(\alpha)^{t}e(-\alpha n)d\alpha.$$
Our goal throughout Sections \ref{sec2} to \ref{sec6} is to obtain a lower bound for $R(n)$ for all sufficiently large $n$.
For such purpose, we make use of a Hardy-Littlewood dissection in our analysis. When $1\leq X\leq M^{k}$, we define the major arcs $\grM(X)$ to be the union of 
\begin{equation}\label{kioto}\grM(a,q)=\Big\{ \alpha\in [0,1): \Big\lvert \alpha-a/q\Big\rvert \leq \frac{X}{qn}\Big\}\end{equation} with $0\leq a\leq q\leq X$ and $(a,q)=1$. For the sake of simplicity we write 
$$\grM=\grM(M^{k}),\ \ \ \ \ \ \ \ \  \ \ \ \ \ \grN=\grM\big((6k)^{-1}H^{1/3}\big).$$We define the minor arcs as $\grm=[0,1)\setminus \grM$ and $\grn=[0,1)\setminus \grN$. This dissection remains valid for the case $k=4$ and will be used in Section \ref{sec7}. We then take $\alpha\in\grm$ and observe that by Dirichlet's approximation there exist non-negative integers $a,q$ with $(a,q)=1$ and $1\leq q\leq nM^{-k}$ such that $$\lvert \alpha-a/q\rvert\leq \frac{M^{k}}{qn}.$$Consequently, one has $q>M^{k}$, and hence (\ref{Up}) and Lemma \ref{lema1} yield the bound
\begin{equation}\label{Wu}W(\alpha)\ll H^{1/2+\varepsilon}M^{1/2}\Big(\sum_{h\leq H} b_{h}^{2}\Big)^{1/2}\ll H^{1+1/24-\tau/6}M^{1/2}.\end{equation}In the following proposition we combine this pointwise bound with some restriction estimates to bound the minor arc contribution.
\begin{prop}\label{prop1}When $k=2,3$ then one has that
\begin{equation}\label{min}\int_{\grm}\lvert h(\alpha)\rvert^{s}\lvert W(\alpha)\rvert^{t}d\alpha\ll (HM)^{t}P^{3s-3k-\delta}.\end{equation}

\end{prop}
\begin{proof}
By Bourgain \cite[(1.6)]{Bou1} when $k=2$ and Hughes and Wooley \cite[Theorem 4.1]{Hug} for the case $k=3$, we find that
$$\int_{0}^{1}\lvert h(\alpha)\rvert^{s}d \alpha\ll P^{3s/2-3k+\varepsilon}\Big(\sum_{x\leq 3P^{3}}a_{x}^{2}\Big)^{s/2}\ll P^{3s-3k+s/8-\delta}.$$ Therefore, an application of the pointwise bound on the minor arcs obtained in (\ref{Wu}) yields the estimate
$$\int_{\grm}\lvert h(\alpha)\rvert^{s}\lvert W(\alpha)\rvert^{t}d \alpha\ll H^{t+t/24}M^{t/2}P^{3s-3k+s/8-\delta}.$$
We define for convenience the parameter $\xi(k)$ as $\xi(2)=0$ and $\xi(3)=7/92$, and deduce that the proposition then follows after noting by (\ref{MH}) that $H^{t/24}M^{t/2}P^{s/8}=M^{t}P^{-\xi(k)}$. For the purpose of this paper, knowing the existence of $\delta>0$ for which (\ref{min}) holds suffices. The reader may observe though that the precise saving over the expected main term that we obtain here is $H^{t\tau/6}P^{\xi(k)+s\tau/2-\varepsilon}$. 
\end{proof}

\section{Approximation of exponential sums over the major arcs}\label{sec3}
We adapt the argument of Vaughan \cite[Theorem 4.1]{Vau} to estimate the difference between the exponential sums $h(\alpha),W(\alpha)$ and their approximations over the major arcs. Let $\mathbf{y}\in [0,P]^{2}$ and set $C_{\mathbf{y}}=y_{1}^{3}+y_{2}^{3}$. Let $\beta\in\mathbb{R}$ and let $p$ be a prime number. Consider the integrals
\begin{equation}\label{esp}v_{\mathbf{y}}(\beta)=\int_{P/2}^{P}e\Big(\beta\big(x^{3}+C_{\mathbf{y}}\big)^{k}\Big)dx\ \ \ \ \text{and}\ \ \ \ v_{\mathbf{y},p}(\beta)=\int_{H_{1}}^{H_{2}}e\Big(\beta p^{3k}\big(x^{3}+C_{\mathbf{y}}\big)^{k}\Big)dx.\end{equation}
Note that by a change of variables one finds that \begin{equation}\label{vyp}v_{\mathbf{y}}(\beta)=\int_{M_{\mathbf{y}}}^{N_{\mathbf{y}}}B_{\mathbf{y}}(\gamma)e(\beta \gamma)d\gamma,\ \ \ \ \ \ v_{\mathbf{y},p}(\beta)=\int_{M_{\mathbf{y},p}}^{N_{\mathbf{y},p}}B_{\mathbf{y},p}(\gamma)e(\beta \gamma)d\gamma,\end{equation}where the limits of integration taken are $M_{\mathbf{y}}=\big(P^{3}/8+C_{\mathbf{y}}\big)^{k},$ $N_{\mathbf{y}}=\big(P^{3}+C_{\mathbf{y}}\big)^{k}$, $M_{\mathbf{y},p}=\big(Hp^{3}/2+C_{p\mathbf{y}}\big)^{k}$ and $N_{\mathbf{y},p}=\big(2Hp^{3}/3+C_{p\mathbf{y}}\big)^{k}$, and the functions inside the integral are defined as
\begin{equation}\label{fun}B_{\mathbf{y}}(\gamma)=\frac{1}{3k}\gamma^{1/k-1}(\gamma^{1/k}-C_{\mathbf{y}})^{-2/3},\ \ \ \ B_{\mathbf{y},p}(\gamma)=\frac{1}{3kp}\gamma^{1/k-1}(\gamma^{1/k}-C_{p\mathbf{y}})^{-2/3}.\end{equation}
We introduce the auxiliary multiplicative function $w_{k}(q)$ defined for prime powers by taking
\begin{equation}\label{wuok}w_{k}(p^{3ku+v})=\left\{
	       \begin{array}{ll}
         p^{-u-v/3k}\ \ \ \ \ \ \text{when $u\geq 1$ and $1\leq v\leq 3k$,}   \\
         p^{-1}\ \ \ \ \ \ \ \  \ \ \ \ \ \text{when $u=0$ and $2\leq v\leq 3k$,}    \\
         p^{-1/2}\ \ \ \ \ \ \  \ \ \ \ \text{when $u=0$ and $v=1$.}    \\
    \end{array}  \right. \end{equation}
In order to discuss the approximation of $f(\alpha)$ on the major arcs, it is convenient to consider for $a\in\mathbb{Z}$ and $q\in\mathbb{N}$ with $(a,q)=1$ the sums
\begin{equation}\label{ggg}S_{\mathbf{y}}(q,a)=\sum_{r=1}^{q}e_{q}\Big(a\big(r^{3}+C_{\mathbf{y}}\big)^{k}\Big)\ \ \ \ \ \text{and}\ \ \ \ \ V(\alpha,q,a)=q^{-1}\sum_{\mathbf{y}}S_{\mathbf{y}}(q,a)v_{\mathbf{y}}(\beta),\end{equation} where $\mathbf{y}$ runs over the set $\mathcal{A}(P,P^{\eta})^{2}$ of pairs of smooth numbers.
\begin{lem}\label{lema2}
Suppose that $a\in\mathbb{Z}$ and $q\in\mathbb{N}$ with $(a,q)=1$. Let $\alpha\in [0,1)$ and $\beta=\alpha-a/q$. Then we have the estimate
$$h(\alpha)-V(\alpha,q,a)\ll P^{2}q^{1+\varepsilon}w_{k}(q)(1+n\lvert\beta\rvert)^{1/2}.$$
Moreover, if $\lvert\beta\rvert\leq (2\cdot 3^{k}kq)^{-1}Pn^{-1}$ one finds that
\begin{equation}\label{faci}h(\alpha)-V(\alpha,q,a)\ll P^{2}q^{1+\varepsilon}w_{k}(q).\end{equation}

\end{lem}
\begin{proof}
Let $b\in\mathbb{N}$ and $\mathbf{y}\in\mathcal{A}(P,P^{\eta})^{2}.$ We define \begin{equation}\label{Syt}S_{\mathbf{y}}(q,a,b)=\sum_{r=1}^{q}e_{q}\big(a\big(r^{3}+C_{\mathbf{y}}\big)^{k}+br\big)\ \ \ \ \ \text{and}\ \ \ \ \ I_{\mathbf{y}}(b)=\int_{P/2}^{P}e\big(F(\gamma;b)\big)d\gamma,\end{equation}where the function in the argument inside the integral is taken to be
$$ F(\gamma;b)=\beta\big(\gamma^{3}+C_{\mathbf{y}})^{k}-b\gamma/q.$$ Both the complete exponential sum and the integral play a role in the analysis of the main and the error term. Observe that $h(\alpha)$ can be written as
$$h(\alpha)=\sum_{\mathbf{y}\in\mathcal{A}(P, P^{\eta})}h_{\mathbf{y}}(\alpha),\ \ \ \ \ \ \ \text{with}\ \ \ \ \ \ \ \ h_{\mathbf{y}}(\alpha)=\sum_{P/2\leq x\leq P}e\big(\alpha(x^{3}+C_{\mathbf{y}})^{k}\big).$$ 
Then by sorting the summation into arithmetic progressions modulo $q$ and applying orthogonality, it follows that
$$h_{\mathbf{y}}(\alpha)=q^{-1}\sum_{-q/2<b\leq q/2}S_{\mathbf{y}}(q,a,b)\sum_{P/2\leq x\leq P}e\big(F(x;b)\big),$$ whence using Vaughan \cite[Lemma 4.2]{Vau} we obtain
\begin{align}\label{hu}h_{\mathbf{y}}(\alpha)-q^{-1}S_{\mathbf{y}}(q,a)v_{\mathbf{y}}(\beta)=
&q^{-1}\sum_{\substack{-B< b\leq B\\ b\neq 0}} S_{\mathbf{y}}(q,a,b)I_{\mathbf{y}}(b)\nonumber
\\
&+O\Big(\log(H+2)q^{-1}\sum_{-q/2< b\leq q/2} \big\lvert S_{\mathbf{y}}(q,a,b)\big\rvert\Big),
\end{align}
 where $B=(H+1/2)q$ and $H=\big\lceil 3^{k}kP^{-1}n\lvert\beta\rvert +1/2\big\rceil.$ 
Note that by the quasi-multiplicative property, in order to bound $S_{\mathbf{y}}(q,a,b)$ it suffices to consider the case when $q$ is a prime power. For such purposes, we take $q=p^{3ku+v}$. We observe first that by Vaughan \cite[Theorem 7.1]{Vau} one has that
$S_{\mathbf{y}}(q,a,b)\ll q^{1-1/3k+\varepsilon}.$ Moreover, when $v\geq 2$ and $u=0$ we can deduce from the proof of the same theorem\footnote{See in particular the argument following \cite[(7.16)]{Vau}} that $S_{\mathbf{y}}(p^{v},a,b)\ll p^{v-1}$. For the case $q=p,$ the work of Weil\footnote{See Schmidt \cite[Corollary 2F]{Sch} for an elementary proof of this bound.} \cite{Wei2} yields the estimate $S_{\mathbf{y}}(p,a,b)\ll p^{1/2}.$ Therefore, combining these bounds with the definition (\ref{wuok}) one finds that
\begin{equation}\label{wee}S_{\mathbf{y}}(q,a,b)\ll q^{1+\varepsilon}w_{k}(q).\end{equation} Consequently, by (\ref{hu}) we have
\begin{equation}\label{hyh}h_{\mathbf{y}}(\alpha)-q^{-1}S_{\mathbf{y}}(q,a)v_{\mathbf{y}}(\beta)\ll q^{\varepsilon}w_{k}(q)\sum_{\substack{-B< b\leq B\\ b\neq 0}}\lvert I_{\mathbf{y}}(b)\rvert+q^{1+\varepsilon}w_{k}(q)\log(H+2).\end{equation}

To treat the sum on the right handside we use the methods of the proof of Vaughan \cite[Theorem 4.1]{Vau}. In his analysis he classifies the range of integration of $I(b)$ according to the size of $\lvert G'(\gamma)\rvert$, where $$G(\gamma)=\beta\gamma^{k}-b\gamma/q\ \ \ \ \ \  \text{and}\ \ I(b)=\displaystyle\int_{0}^{X}e\big(G(\gamma)\big)d\gamma.$$ We follow Vaughan's analysis closely, dividing the range of integration of $I_{\mathbf{y}}(b)$ according to the size of $\lvert F'(\gamma;b)\rvert,$ to obtain $$\sum_{\substack{-B< b\leq B\\ b\neq 0}}\lvert I_{\mathbf{y}}(b)\rvert\ll q^{1+\varepsilon}(1+n\lvert \beta\rvert )^{1/2}.$$ Since $\log(H+2)\ll (1+n\lvert\beta\rvert)^{1/2}$ then
\begin{equation*}h_{\mathbf{y}}(\alpha)-q^{-1}S_{\mathbf{y}}(q,a)v_{\mathbf{y}}(\beta)\ll q^{1+\varepsilon}w_{k}(q)(1+n\lvert\beta\rvert)^{1/2},\end{equation*}
which implies the first statement of the lemma by summing over $\mathbf{y}\in\mathcal{A}(P,P^{\eta})^{2}$.
Note that when $\lvert\beta\rvert\leq (2\cdot 3^{k}kq)^{-1}Pn^{-1}$ and $b\neq 0$ one has $\lvert F'(x;b)\rvert\geq \lvert b\rvert/2q$ and $H=1$. Observing that $F'(x;b)$ is monotonous then partial integration yields $$\sum_{\substack{-B< b\leq B\\ b\neq 0}}\lvert I_{\mathbf{y}}(b)\rvert\ll \sum_{\substack{-B< b\leq B\\ b\neq 0}}\frac{q}{\lvert b\rvert}\ll q^{1+\varepsilon}.$$ Combining this estimate with (\ref{hyh}) and summing over $\mathbf{y}\in\mathcal{A}(P,P^{\eta})^{2}$ we get (\ref{faci}).
\end{proof}
By applying similar methods we can obtain the same type of approximation for the exponential sum $W(\alpha).$ For $a\in\mathbb{Z}$ and $q\in\mathbb{N}$ with $(a,q)=1$ and recalling (\ref{esp}) and (\ref{ggg}) we introduce the auxiliary function \begin{equation}\label{Wa}W(\alpha,q,a)=q^{-1}\sum_{\mathbf{y},p}S_{p\mathbf{y}}(q,a)v_{\mathbf{y},p}(\beta),\ \ \ \ \text{where}\ \mathbf{y}\in\mathcal{A}(H_{3},P^{\eta})^{2}\ \text{and}\  M/2\leq p\leq M.\end{equation}
\begin{lem}\label{lema3}
Suppose that $(a,q)=1$ and $(p,q)=1$ for all primes with $M/2\leq p\leq M$. Let $\alpha\in [0,1)$ and $\beta=\alpha-a/q$. Then we have the estimate 
$$W(\alpha)-W(\alpha,q,a)\ll MH^{2/3}q^{1+\varepsilon}w_{k}(q)(1+n\lvert\beta\rvert)^{1/2}(\log P)^{-1}.$$
Moreover, if $\lvert\beta\rvert\leq (6kq)^{-1}H^{1/3}n^{-1}$ one finds that
$$W(\alpha)-W(\alpha,q,a)\ll MH^{2/3}q^{1+\varepsilon}w_{k}(q)(\log P)^{-1}.$$
\end{lem}
\begin{proof}In the same way as before, we can express the exponential sum $W(\alpha)$ as
$$W(\alpha)=\sum_{\substack{\mathbf{y},p}}W_{\mathbf{y},p}(\alpha),\ \ \ \ \ \ \ \text{where}\ \  \ W_{\mathbf{y},p}(\alpha)=\sum_{H_{1}\leq x\leq H_{2}}e\big(\alpha p^{3k}(x^{3}+C_{\mathbf{y}})^{k}\big).$$ 
Sorting the summation into arithmetic progressions modulo $q$ and applying orthogonality one has that
$$W_{\mathbf{y},p}(\alpha)=q^{-1}\sum_{-q/2<b\leq q/2}S_{\mathbf{y}}(q,ap^{3k},b)\sum_{H_{1}\leq x\leq H_{2}}e\Big(\beta p^{3k}(x^{3}+C_{\mathbf{y}})^{k}-\frac{bx}{q}\Big).$$ Recalling that $(q,p)=1$ then a change of variables yields $S_{\mathbf{y}}(q,ap^{3k})=S_{p\mathbf{y}}(q,a).$ Therefore, the application of the argument of Vaughan \cite[Theorem 4.1]{Vau} in the same way as we did above leads to
\begin{equation*}\label{Wy}W_{\mathbf{y},p}(\alpha)-q^{-1}S_{p\mathbf{y}}(q,a)v_{\mathbf{y},p}(\beta)\ll q^{1+\varepsilon}w_{k}(q)(1+n\lvert\beta\rvert)^{1/2},\end{equation*} and if $\lvert\beta\rvert\leq (6kq)^{-1}H^{1/3}n^{-1}$ then $$W_{\mathbf{y},p}(\alpha)-q^{-1}S_{p\mathbf{y}}(q,a)v_{\mathbf{y},p}(\beta)\ll q^{1+\varepsilon}w_{k}(q),$$
which delivers the desired result by summing over the range of $(\mathbf{y},p)$ described in (\ref{Wa}).
\end{proof}

\section{Treatment of the singular series}\label{sec4}
Unless specified, in this section and the two upcoming ones we assume that $k=2,3.$ We introduce some exponential sums and present upper bounds which we obtain making use of the arguments in Vaughan \cite[Theorem 7.1]{Vau}. We also discuss the congruence problem and introduce some divisibility constraints on $C_{\mathbf{y}_{i}}$ and $C_{p_{i}\mathbf{y}_{i}}$ to ensure local solubility. For the rest of the paper, unless specified, $\mathbf{Y}=(\mathbf{y}_{1},\ldots, \mathbf{y}_{s+t})\in \mathbb{N}^{2s+2t}$ and $\mathbf{p}=(p_{1},\dots,p_{t})$ will denote tuples with $\mathbf{y}_{i}\in\mathcal{A}(P,P^{\eta})^{2}$ for $t+1\leq i\leq s+t$ and $\mathbf{y}_{i}\in\mathcal{A}(H_{3},P^{\eta})^{2}$ for $1\leq i\leq t$, where $p_{i}$ are primes satisfying $M/2\leq p_{i}\leq M.$ Take $q\in\mathbb{N}$ and define \begin{equation*}\label{Syp}S_{\mathbf{Y},\mathbf{p}}(q)=q^{-s-t}\sum_{\substack{a=1\\ (a,q)=1}}^{q}e(-an/q)\prod_{i=1}^{t} S_{p_{i}\mathbf{y}_{i}}(q,a)\prod_{i=t+1}^{s+t} S_{\mathbf{y}_{i}}(q,a).\end{equation*}The following technical lemma provides a straightforward upper bound for the previous exponential sum and will be used throughout the major arc treatment.
\begin{lem}\label{expo}
Assume that $2\leq k\leq 4.$ Let $m\geq 2$. Take $\alpha\leq\frac{m-1}{3k}$ when $m\geq 3$ and $\alpha=0$ for $m=2$. Let $Q\geq 1$. Then, recalling (\ref{wuok}) one has 
$$\sum_{q\leq Q}q^{\alpha}w_{k}(q)^{m}\ll Q^{\varepsilon}.$$Moreover, for the case $k=4$ we also have
\begin{equation}\label{sum}\sum_{q\leq Q}q\tau_{4}(q)^{4}w_{4}(q)\ll Q^{\varepsilon},\end{equation} where $\tau_{4}(q)$ was defined just before Lemma \ref{lema1}.
\end{lem} 
\begin{proof} By the multiplicative property of $w_{k}(q)$ it follows that
$$\sum_{q\leq Q}q^{\alpha}w_{k}(q)^{m}\ll \prod_{p\leq Q}\big(1+\sum_{h=1}^{\infty}p^{h\alpha}w_{k}(p^{h})^{m}\big)\ll \prod_{p\leq Q}\big(1+Cp^{-1}\big)\ll Q^{\varepsilon}.$$ For the second estimate we use the bound $\tau_{4}(p^{h})^{4}\ll p^{-h}$ when $h\geq 2$ to obtain
$$\sum_{h=1}^{\infty}p^{h}\tau_{4}(p^{h})^{4}w_{4}(p^{h})\ll p^{-3/2}+\sum_{h\geq 2}w_{4}(p^{h})\ll p^{-1}.$$ Equation (\ref{sum}) follows then combining the above bound with multiplicativity.
\end{proof}
\begin{lem}\label{lema5}
Let $a\in\mathbb{Z}$ and $q\in\mathbb{N}$ with $(a,q)=1$. The functions $S_{\mathbf{y}}(q,a)$ and $S_{\mathbf{Y},\mathbf{p}}(q)$ defined above satisfy
\begin{equation}\label{Ss}S_{\mathbf{y}}(q,a)\ll q^{1+\varepsilon}w_{k}(q),\ \ \ \ \ \ \ \ \ \ \ \  \ S_{\mathbf{Y},\mathbf{p}}(q)\ll q^{1+\varepsilon}w_{k}(q)^{s+t}.\end{equation}
As a consequence, for every $Q\geq 1$ and every $\alpha\leq\frac{s+t-1}{3k}-1$ it follows that \begin{equation}\label{qal}\sum_{\substack{q\leq Q}}q^{\alpha}\lvert S_{\mathbf{Y},\mathbf{p}}(q)\rvert\ll Q^{\varepsilon}\ \ \ \text{and}\ \ \ \ \sum_{\substack{q>Q}}\lvert S_{\mathbf{Y},\mathbf{p}}(q)\rvert\ll Q^{\varepsilon-\alpha}.\end{equation}
\end{lem}
\begin{proof}
On recalling (\ref{Syt}) note that $S_{\mathbf{y}}(q,a)=S_{\mathbf{y}}(q,a,0)$. Therefore, (\ref{wee}) yields $S_{\mathbf{y}}(q,a)\ll q^{1+\varepsilon}w_{k}(q)$, and hence (\ref{Ss}) holds. This estimate and Lemma \ref{expo} imply the first inequality in (\ref{qal}). Finally, observe that as a consequence we have
$$\sum_{\substack{Q\leq q\leq 2Q}}\lvert S_{\mathbf{Y},\mathbf{p}}(q)\rvert\ll Q^{\varepsilon-\alpha},$$ from where the second inequality of (\ref{qal}) follows by summing over dyadic intervals.
\end{proof}
We apply the bounds obtained in the previous lemma to a collection of singular series and other related series. For such purpose, it is convenient to define, for tuples $(\mathbf{Y}, \mathbf{p})$ and each prime $p$ the sums
$$\frak{S}_{\mathbf{Y},\mathbf{p}}(n)=\sum_{q=1}^{\infty}S_{\mathbf{Y},\mathbf{p}}(q),\ \ \ \ \ \ \ \ \ \ \ \ \ \ \ \ \sigma(p)=\sum_{l=0}^{\infty}S_{\mathbf{Y},\mathbf{p}}(p^{l}).$$
\begin{lem}\label{lema6}
The singular series $\frak{S}_{\mathbf{Y},\mathbf{p}}(n)$ converges absolutely, the identity
\begin{equation}\label{Sin}\frak{S}_{\mathbf{Y},\mathbf{p}}(n)=\prod_{p}\sigma(p)\end{equation}
holds and $0\leq\frak{S}_{\mathbf{Y},\mathbf{p}}(n)\ll 1$. Furthermore, one has $\frak{S}_{\mathbf{Y},\mathbf{p}}(n)\gg 1$ provided that:
\begin{enumerate}
\item When $k=2$ one has $C_{p_{i}\mathbf{y}_{i}}\equiv 28 \pmod{108}$ for $1\leq i\leq t$ and $C_{\mathbf{y}_{i}}\equiv 28\pmod{108}$ for $t+1\leq i\leq s+t$;
\item When $k=3$ one has $C_{p_{i}\mathbf{y}_{i}}\equiv 0 \pmod{162}$ for $1\leq i\leq t$ and $C_{\mathbf{y}_{i}}\equiv 0\pmod{162}$ for $t+1\leq i\leq s+t$.
\end{enumerate}
\end{lem}
As mentioned before, the constraints on $C_{\mathbf{y}_{i}}$ and $C_{p_{i}\mathbf{y}_{i}}$ ensure the local solubility of the problem. Note that the set of tuples with these divisibility conditions has positive density over the set of tuples without the restrictions since it follows from the proof of Lemma 5.4 of \cite{Vau3} that smooth numbers are well distributed on arithmetic progressions. Therefore, we are still able to get the expected lower bound for the major arc contribution. Observe though that the choices for the constraints are not unique, but for the purpose of this exposition it will suffice to study just one of the possible restrictions.
\begin{proof}
Note that the application of Lemma \ref{lema5} yields the estimate \begin{equation}\label{sig}\sigma(p)-1\ll p^{-2}.\end{equation} This bound and the multiplicative property of $S_{\mathbf{Y},\mathbf{p}}(q)$ imply (\ref{Sin}), the convergence of the series $\frak{S}_{\mathbf{Y},\mathbf{p}}(n)$ and its upper bound.
To give a more arithmetic description of $\sigma(p)$ it is convenient to introduce 
$$\mathcal{M}_{n}(p^{h})=\Big\{\mathbf{X}\in [1,p^{h}]^{s+t}:\ \ n\equiv\sum_{i=1}^{t}(x_{i}^{3}+C_{p_{i}\mathbf{y}_{i}})^{k}+\sum_{i=t+1}^{s+t}(x_{i}^{3}+C_{\mathbf{y}_{i}})^{k} \pmod{p^{h}}\Big\}$$ and $M_{n}(p^{h})=\lvert \mathcal{M}_{n}(p^{h})\rvert.$ Observe that by a standard argument making use of orthogonality we obtain the relation
$$\sum_{l=0}^{h}S_{\mathbf{Y},\mathbf{p}}(p^{l})=p^{(1-s-t)h}M_{n}(p^{h}).$$ 
In view of (\ref{sig}) it transpires then that in order to prove the lower bound for $\mathfrak{S}_{\mathbf{Y},\mathbf{p}}(n)$ it will suffice to show that $p^{(1-s-t)h}M_{n}(p^{h})\geq C_{p}$ for some positive constant $C_{p}$ depending on $p$. For each $p$ prime, take $\tau\geq 0$ for which $p^{\tau}\| 3k$. Define $\gamma=\gamma(p)=2\tau+1$ and $$\mathcal{M}_{n}^{*}(p^{\gamma})=\Big\{\mathbf{X}\in\mathcal{M}_{n}(p^{\gamma}):\ \ p\nmid x_{1},\ \ p\nmid (x_{1}^{3}+C_{p_{1}\mathbf{y}_{1}})\Big\}.$$ We take $h\geq \gamma$ for convenience. Our priority for the rest of the lemma will be to show that $\lvert\mathcal{M}_{n}^{*}(p^{\gamma})\rvert>0$, since then an application of Hensel's Lemma will yield the bound $M_{n}(p^{h})\geq p^{(s+t-1)(h-\gamma)}.$ 

For further discussion, it is convenient to consider for a fixed number $C\in\mathbb{N}$ the sets $$\mathcal{T}_{C}(p^{\gamma})=\Big\{x^{3}+C\pmod{p^{\gamma}}\Big\},\ \  \mathcal{T}_{C}^{*}(p^{\gamma})=\Big\{x^{3}+C\pmod{p^{\gamma}}:\ p\nmid x,\ \ p\nmid (x^{3}+C)\Big\}.$$ Let $p\equiv 1\pmod{3}.$ Under this condition one has $p\geq 7$, so $\gamma=1$ with $\lvert\mathcal{T}_{C}(p)\rvert=(p+2)/3$ and $\lvert\mathcal{T}_{C}^{*}(p)\rvert\geq 1.$ If we denote the set of $k$-th powers of the above set by $$\mathcal{T}_{C}^{k}(p^{\gamma})=\Big\{y^{k} \pmod{p^{\gamma}}:\ y\in\mathcal{T}_{C}(p^{\gamma})\Big\},$$ then one finds that $\lvert\mathcal{T}_{C}^{k}(p)\rvert\geq \big\lceil(p+2)/3k\big\rceil.$ One can check that $\lvert\mathcal{T}_{C}^{k}(7)\rvert\geq 2$ for every $C\in\mathbb{N}$, and whenever $p>7$ we find that
$$(s+t-1)\Big(\Big\lceil\frac{p+2}{3k}\Big\rceil-1\Big)\geq p,$$
and hence Cauchy-Davenport delivers $\lvert\mathcal{M}_{n}^{*}(p)\lvert>0$. 
When $p\equiv 2\pmod{3}$ and $p>2$ then $\gamma=1$ and we further get $\lvert\mathcal{T}_{C}(p)\rvert=p$ and $\lvert\mathcal{T}_{C}^{*}(p)\rvert\geq 1$, whence another application of Cauchy-Davenport yields $\lvert\mathcal{M}_{n}^{*}(p)\rvert>0$. For the case $p=2$ the divisibility contraints reduce the problem to the resolution of
$$y_{1}^{6}+\dots+y_{8}^{6}\equiv n\pmod{8}$$ with $y_{i}\in\mathbb{N}$ and $2\nmid y_{1},$ which is straightforward. The case $k=3$ is also trivial since then one would have $\gamma(2)=1$. Likewise, if $p=3$ one finds that whenever $C\equiv 1\mmod{27}$ then $\mathcal{T}_{C}^{2}(27)=\{0,1,4,13,22\}$ and $\lvert\mathcal{T}_{C}^{*}(27)\rvert=3,$ so $\lvert\mathcal{M}_{n}^{*}(27)\rvert>0$ when $k=2$ follows combining the constraints for $C_{p_{i}\mathbf{y}_{i}}$ and $C_{\mathbf{y}_{i}}$ described above and Vaughan \cite[Lemma 2.14]{Vau}. Finally, when $k=3$ we make use of the conditions $ C_{\mathbf{y}_{i}}\equiv 0\mmod{81}$ and $C_{p_{i}\mathbf{y}_{i}}\equiv 0\mmod{81}$ to reduce the problem to finding a solution for
$$y_{1}^{9}+\ldots+y_{17}^{9}\equiv n\mmod{243}$$ with $y_{i}\in\mathbb{N}$ and $3\nmid y_{1}$. The solubility of this congruence is again a consequence of Vaughan \cite[Lemma 2.14]{Vau}. 
\end{proof}

\section{Singular integral}\label{sec5}
In this section we analyse the size of the singular integral following the classical approach making use of Fourier's Integral Theorem. For each pair of tuples $(\mathbf{Y},\mathbf{p})$ consider \begin{equation*}\label{Jnn}J_{\mathbf{Y},\mathbf{p}}(n)=\int_{-\infty}^{\infty}V_{\mathbf{Y},\mathbf{p}}(\beta)e(-n\beta)d\beta,\ \ \ \ \text{where}\ \ \ \\ V_{\mathbf{Y},\mathbf{p}}(\beta)=\prod_{i=1}^{t}v_{\mathbf{y}_{i},p_{i}}(\beta)\prod_{i=t+1}^{s+t}v_{\mathbf{y}_{i}}(\beta),\end{equation*} and $v_{\mathbf{y_{i}},p_{i}}(\beta)$ and $v_{\mathbf{y}_{i}}(\beta)$ were defined in (\ref{esp}).
\begin{lem}\label{lemi}
One has that $0\leq J_{\mathbf{Y},\mathbf{p}}(n)\ll P^{s}H^{t/3}n^{-1}$. Moreover, whenever $(\mathbf{Y},\mathbf{p})$ satisfies $M/2\leq p_{i}\leq 51 M/100$ for $1\leq i\leq t$ and $\mathbf{y}_{i}\leq P/2$ for $t+1\leq i\leq s+t$ then 
\begin{equation}\label{jn}J_{\mathbf{Y},\mathbf{p}}(n)\gg P^{s}H^{t/3}n^{-1}.\end{equation}
\end{lem}
In the following discussion we rewrite $J_{\mathbf{Y},\mathbf{p}}(n)$ as an integral whose size is easier to estimate. The conditions on the tuples described before ensure that we get a suitable range of integration for such integral. Note that the set of tuples on that range has positive density over the set of tuples without the restrictions, and hence we are still able to get the expected lower bound for the major arc contribution.
\begin{proof}
By using the expression of both $v_{\mathbf{y}}(\beta)$ and $v_{\mathbf{y},p}(\beta)$ in (\ref{vyp}) we find that
$$J_{\mathbf{Y},\mathbf{p}}(n)=\lim_{\lambda\rightarrow\infty}\int_{-\lambda}^{\lambda}\int_{\mathbf{x}\in\mathcal{S}}B_{\mathbf{Y},\mathbf{p}}(\mathbf{x})e\Big(\beta\Big(\sum_{i=1}^{s+t}x_{i}-n\Big)\Big)d\mathbf{x}d\beta,$$
where the function $B_{\mathbf{Y},\mathbf{p}}(\mathbf{x})$ is taken to be $$B_{\mathbf{Y},\mathbf{p}}(\mathbf{x})=\prod_{i=1}^{t}B_{\mathbf{y}_{i},p_{i}}(x_{i})\prod_{i=t+1}^{s+t}B_{\mathbf{y}_{i}}(x_{i})$$ and we integrate over the set $\mathcal{S}=\prod [M_{\mathbf{y}_{i},\mathbf{p}_{i}},N_{\mathbf{y}_{i},\mathbf{p}_{i}}]\times \prod[M_{\mathbf{y}_{i}},N_{\mathbf{y}_{i}}]$.
Then by integrating on $\beta$ and making the change of variables $v=\sum_{i=1}^{s+t}x_{i}$ we obtain
$$J_{\mathbf{Y},\mathbf{p}}(n)=\lim_{\lambda\rightarrow\infty}\int_{S_{1}}^{S_{2}}\phi(v)\frac{\sin\big(2\pi\lambda(v-n)\big)}{\pi(v-n)}dv,$$
where $\phi(v)$ is defined as $$\phi(v)=\int_{\mathbf{x}\in\mathcal{S}'(v)}B_{\mathbf{y}_{s+t}}\Big(v-\sum_{i=1}^{s+t-1}x_{i}\Big)\prod_{i=1}^{t}B_{\mathbf{y}_{i},p_{i}}(x_{i})\prod_{i=t+1}^{s+t-1}B_{\mathbf{y}_{i}}(x_{i})d\mathbf{x}$$ and $\mathcal{S}'(v)\subset \mathbb{R}^{s+t-1}$ denotes the subset of tuples satisfying $$x_{i}\in [M_{\mathbf{y}_{i},p_{i}}, N_{\mathbf{y}_{i},p_{i}}]\ \ \ \ \text{for $1\leq i\leq t$,} \ \ \ \ \ x_{i}\in [M_{\mathbf{y}_{i}},N_{\mathbf{y}_{i}}]\ \ \ \ \text{for $t+1\leq i\leq s+t-1$},$$ and 
\begin{equation}\label{MNMN}M_{\mathbf{y}_{s+t}}\leq v-\sum_{i=1}^{s+t-1}x_{i}\leq N_{\mathbf{y}_{s+t}}.\end{equation}
Since $\phi(v)$ is a function of bounded variation, it follows from Fourier's Integral Theorem that
$J_{\mathbf{Y},\mathbf{p}}(n)=\phi(n)$, which implies positivity. Note that combining the identity $P^{3}=M^{3}H$, which is a consequence of (\ref{MH}), the limits of integration defined after (\ref{vyp}) and equation (\ref{fun}), we find that whenever $\mathbf{x}\in\mathcal{S}'(n)$ then it follows that $B_{\mathbf{y}_{i},p_{i}}(x_{i})\asymp H^{1/3}n^{-1}$ for $1\leq i\leq t$ and $B_{\mathbf{y}_{i}}(x_{i})\asymp Pn^{-1}$ for $t+1\leq i\leq s+t-1$, and one further has
$$B_{\mathbf{y}_{s+t}}(n-\sum_{i=1}^{s+t-1}x_{i})\asymp Pn^{-1}.$$
Therefore, combining the previous ideas we obtain the upper bound for $J_{\mathbf{Y},\mathbf{p}}(n)$ stated at the beginning of the lemma. Moreover, if $(\mathbf{Y},\mathbf{p})$ lies in the range described right after that bound, then there exist intervals $I_{i}\subset [M_{\mathbf{y}_{i},p_{i}}, N_{\mathbf{y}_{i},p_{i}}]$ for $1\leq i\leq t$ and $I_{i}\subset [M_{\mathbf{y}_{i}},N_{\mathbf{y}_{i}}]$ for $t+1\leq i\leq s+t-1$ satisfying $|I_{i}|\asymp n$ and with the property that whenever $x_{i}\in I_{i}$ then (\ref{MNMN}) holds for $v=n$. Consequently, the preceding discussion yields (\ref{jn}).
\end{proof}
For the sake of brevity we define the auxiliary functions $h^{*}(\alpha)$ and $W^{*}(\alpha)$ by putting
\begin{equation*}\label{ec6.2}h^{*}(\alpha)=V(\alpha,q,a)\ \ \ \ \ \text{and}\ \ \ \ \ W^{*}(\alpha)=W(\alpha,q,a) \end{equation*} when $\alpha\in\grM(a,q)\subset\grM$ and $h^{*}(\alpha)=W^{*}(\alpha)=0$ for $\alpha\in\grm.$ Here the reader may want to recall (\ref{ggg}) and (\ref{Wa}). For the rest of the section we present some bounds for these functions.
\begin{lem}\label{lema4}
Let $\beta\in\mathbb{R}.$ For every prime $p$ and $\mathbf{y}\in \mathbb{N}^{2}$ one has
$$v_{\mathbf{y}}(\beta)\ll \frac{P}{1+n\lvert \beta\rvert}\ \ \ \ \ \text{and}\ \ \ \ \ v_{\mathbf{y},p}(\beta)\ll \frac{H^{1/3}}{1+n\lvert \beta\rvert}.$$Moreover, whenever $\alpha\in \grM(a,q)\subset \grM$ one finds that
\begin{equation*}h^{*}(\alpha)\ll \frac{q^{\varepsilon}w_{k}(q)P^{3}}{1+n\lvert\alpha-a/q\rvert}\ \ \ \ \text{and}\ \ \ \ \ W^{*}(\alpha)\ll \frac{q^{\varepsilon}w_{k}(q)MH}{(1+n\lvert\alpha-a/q\rvert)(\log P)}.\end{equation*}
\end{lem}
\begin{proof}
When $\lvert \beta\rvert\leq n^{-1}$, the bound for $v_{\mathbf{y}}(\beta)$ follows observing that by (\ref{vyp}) and the limits of integration taken after (\ref{vyp}) then
$$v_{\mathbf{y}}(\beta)\ll \int_{M_{\mathbf{y}}}^{ N_{\mathbf{y}}}y^{1/k-1}\big(y^{1/k}-C_{\mathbf{y}}\big)^{-2/3}dy\ll P.$$ For the case $\lvert \beta\rvert> n^{-1},$ using the fact that $B_{\mathbf{y}}(y)$ is decreasing and integrating by parts we have that
$$v_{\mathbf{y}}(\beta)\ll \lvert\beta\rvert^{-1}B_{\mathbf{y}}(M_{\mathbf{y}})\ll \lvert\beta\rvert^{-1}n^{1/3k-1},$$which proves the statement. The case $v_{\mathbf{y},p}(\beta)$ is done in a similar way and follows after applying the identity $P^{3}=M^{3}H,$ which is a consequence of (\ref{MH}). Combining these estimates and Lemma \ref{lema5} we get the bounds for $h^{*}(\alpha)$ and $W^{*}(\alpha)$.
\end{proof}
\section{Major arc contribution}\label{sec6}
In this section we show that the contribution of the set of narrow arcs $\grN$ is asymptotic to the expected main term. We prove then that the contribution of the remaining arcs is smaller by combining major and minor arc techniques and making use of Lemma \ref{lema1}. 
\begin{prop}\label{prop2}
There exists $\delta>0$ such that
\begin{equation*}\label{Wh}\int_{\grM}h(\alpha)^{s}W(\alpha)^{t}e(-\alpha n)\alpha=\sum_{\mathbf{Y},\mathbf{p}}\mathfrak{S}_{\mathbf{Y},\mathbf{p}}(n)J_{\mathbf{Y},\mathbf{p}}(n)+O(H^{t}M^{t}P^{3s-3k-\delta}).\end{equation*}
\end{prop}
\begin{proof}
We note first that the triangle inequality yields
$$h(\alpha)^{s}-h^{*}(\alpha)^{s}\ll \lvert h(\alpha)-h^{*}(\alpha)\rvert\big(\lvert h^{*}(\alpha)\rvert^{s-1}+\lvert h(\alpha)-h^{*}(\alpha)\rvert^{s-1}\big).$$
Observe that by (\ref{MH}) and the definition (\ref{kioto}) then whenever $\alpha\in\grN(a,q)$ one has that $(1+n\lvert\beta\rvert)^{-1}\geq qH^{-1/3}\geq qP^{-1}$ and $\lvert\beta\rvert\leq (6kq)^{-1}H^{1/3}n^{-1}\leq (2\cdot 3^{k}kq)^{-1}Pn^{-1}$ for $n$ sufficiently large. Consequently, Lemma \ref{lema2} applied to $\lvert h(\alpha)-h^{*}(\alpha)\rvert$ and Lemma \ref{lema4} applied to $\lvert h^{*}(\alpha)\rvert$ in the above equation deliver
\begin{equation}\label{ech}
h(\alpha)^{s}-h^{*}(\alpha)^{s}\ll q^{1+\varepsilon}w_{k}(q)^{s}P^{3s-1}(1+n\lvert\beta\rvert)^{-s+1},
\end{equation}and by the same reason then whenever $\alpha\in\grN(a,q)$ with $(p,q)=1$ for all primes $M/2\leq p\leq M$, Lemma \ref{lema3} gives
\begin{equation}\label{Wal}
W(\alpha)^{t}-W^{*}(\alpha)^{t}\ll M^{t}H^{t-1/3}q^{1+\varepsilon}w_{k}(q)^{t}(1+n\lvert\beta\rvert)^{-t+1}.
\end{equation}
We also need a bound on the following quantity to exploit some orthogonality relation when averaging over $q$. Denote by $N(q,P)$ the number of solutions of the congruence 
$$T(p_{1}\mathbf{x}_{1})^{k}+T(p_{2}\mathbf{x}_{2})^{k}\equiv T(p_{3}\mathbf{x}_{3})^{k}+T(p_{4}\mathbf{x}_{4})^{k}\pmod {q},$$ where $\mathbf{x}_{i}\in [1,H^{1/3}]^{3}$ and $M/2\leq p_{i}\leq M$ with $q\in\mathbb{N}$. By expressing $q$ as the product of prime powers, using the structure of the ring of integers of these prime powers and noting that the number of primes dividing $q$ is $O\big((\log q)/\log\log q\big)$ we obtain
\begin{equation}\label{NqP}N(q,P)\ll q^{\varepsilon}(MH)^{4}(\log P)^{-4}\big(q^{-1}+P^{-1}\big),\end{equation}where we also used the identity $P=MH^{1/3}$, and hence by orthogonality it follows that
\begin{equation}\label{Wsa}\sum_{a=1}^{q}\lvert W(\beta+a/q)\rvert^{4}\leq qN(q,P)\ll q^{1+\varepsilon}(MH)^{4}(\log P)^{-4}\big(q^{-1}+P^{-1}\big).\end{equation} Combining (\ref{ech}) and (\ref{Wsa}) one has that
\begin{align*}\int_{\grN}\Big\lvert h(\alpha)^{s}-h^{*}(\alpha)^{s}\Big\rvert \lvert W(\alpha)\rvert^{t}d \alpha&
\ll (HM)^{t}P^{3s-3k-1}\sum_{q\leq H^{1/3}}q^{1+\varepsilon}w_{k}(q)^{s}
\\
&\ll (HM)^{t}P^{3s-3k-\delta},
\end{align*}
where we used (\ref{MH}) and Lemma \ref{expo}. Before introducing the auxiliary function $W^{*}(\alpha)$ to replace $W(\alpha)$ we must ensure that the contribution of the arcs with $M/4<q\leq (6k)^{-1}H^{1/3}$ is small enough. By doing so we avoid having to approximate $W(\alpha)$ for the cases when $p\mid q$ for primes $p$ appearing in the definition (\ref{Weig}) of $W(\alpha)$. Combining Lemma \ref{lema4} with (\ref{Wsa}) one finds that \begin{align*}\sum_{M/4<q\leq (6k)^{-1}H^{1/3}}&\sum_{\substack{a=1\\ (a,q)=1}}^{q}\int_{0}^{1}\lvert h^{*}(\beta+a/q)\rvert^{s}\big\lvert W(\beta+a/q)\big\rvert^{t}d \beta\nonumber
\\
\ll &(HM)^{t}P^{3s-3k+\varepsilon}\sum_{M/4<q\leq (6k)^{-1}H^{1/3}}w_{k}(q)^{s}\ll (HM)^{t}P^{3s-3k-\delta},
\end{align*}
where in the last step we applied the definition (\ref{wuok}). For the range $q\leq M/4$ we always have $(p,q)=1$ for all primes $M/2\leq p\leq M$, so we can use (\ref{Wal}) and Lemma \ref{lema4} to obtain
\begin{align*}\sum_{q\leq M/4}\sum_{\substack{a=1\\ (a,q)=1}}^{q}&
\int_{\grN(a,q)}\lvert h^{*}(\alpha)\rvert^{s}\Big\lvert W(\alpha)^{t}-W^{*}(\alpha)^{t}\Big\rvert d \alpha
\\
&\ll P^{3s-3k}M^{t}H^{t-1/3}\sum_{q\leq M/4}q^{2+\varepsilon}w_{k}(q)^{s+t}\ll (HM)^{t}P^{3s-3k-\delta},
\end{align*} where in the last line we used (\ref{MH}) and applied Lemma \ref{expo}.
By Lemmata \ref{expo} and \ref{lema4} one has that
\begin{align*}
\sum_{q\leq M/4}\sum_{\substack{a=1\\ (a,q)=1}}^{q}&
\int_{\lvert\alpha-a/q\rvert>(6kq)^{-1}H^{1/3}n^{-1}}\lvert h^{*}(\alpha)\rvert^{s}\lvert W^{*}(\alpha)\rvert^{t}d\alpha
\\
&\ll H^{2t/3-s/3+1/3}M^{t}P^{3s-3k}\sum_{q\leq M/4}q^{s+t+\varepsilon}w_{k}(q)^{s+t}\ll (HM)^{t}P^{3s-3k-\delta}.
\end{align*}
Therefore, using the previous bounds, making a change of variables and combining Lemmata \ref{lema5} and \ref{lema4} it follows that
\begin{align}\label{hWW}\int_{\grN} h(\alpha)^{s} W(\alpha)^{t}e(-\alpha n)d\alpha=\sum_{\mathbf{Y},\mathbf{p}}\frak{S}_{\mathbf{Y},\mathbf{p}}(n)J_{\mathbf{Y},\mathbf{p}}(n)+O\big((HM)^{t}P^{3s-3k-\delta}\big).
\end{align} 

The rest of the section is devoted to ensure that the contribution of the remaining major arcs is smaller than the main term in the previous equation. Let $R(q,P)$ be the number of solutions of the congruence 
$$T(\mathbf{x}_{1})^{k}+T(\mathbf{x}_{2})^{k}\equiv T(\mathbf{x}_{3})^{k}+T(\mathbf{x}_{4})^{k}\pmod {q},$$ where $\mathbf{x}_{i}\in [1,P]^{3}$. Applying the same argument we used in (\ref{NqP}) for bounding $N(q,P)$ we find that
$R(q,P)\ll q^{\varepsilon}P^{12}\max(q^{-1},P^{-1}),$ and hence by orthogonality it follows that \begin{equation}\label{Wint}\sum_{a=1}^{q}\lvert h(\beta+a/q)\rvert^{4}\leq qR(q,P)\ll q^{1+\varepsilon}P^{12}\max(q^{-1},P^{-1}).\end{equation} 
Moreover, observe that by a similar argument for the case $k=2$ we get
\begin{equation}\label{squa}\sum_{a=1}^{q}\lvert h(\beta+a/q)\rvert^{2}\ll q^{1+\varepsilon}P^{6}\max(q^{-1},P^{-1}).\end{equation}

We consider for convenience the mean value $$I_{M}=\int_{\grM\setminus\grN}\lvert h(\alpha)\rvert^{s}\big\lvert W(\alpha)\big\rvert^{t}d \alpha.$$Our strategy for the treatment of this integral will be to bound $W(\alpha)$ pointwise via Lemma \ref{lema1} and use some major arc estimates. For such purposes, we define first $\Upsilon(\alpha)$ for $\alpha\in [0,1)$ by taking
$$\Upsilon(\alpha)=\tau_{k}(q)(1+n\lvert\alpha-a/q\lvert)^{-1}$$ when $\alpha\in\grM(a,q)\subset\grM$ and $\Upsilon(\alpha)=0$ otherwise. When $a\in\mathbb{Z}$ and $q\in\mathbb{N}$ satisfy $0\leq a\leq q\leq M^{k}$ and $(a,q)=1$, consider the set of arcs
\begin{equation}\label{kiotos}\grM'(a,q)=\Big\{ \alpha\in [0,1): \Big\lvert \alpha-a/q\Big\rvert \leq \frac{M}{q^{1/k}n}\Big\}\end{equation}and take $\grM'$ to be the union of such arcs. Note that then one has $\grM'\subset \grM.$ Observe that for $\alpha\in \grM\setminus\grM',$ the bound in the right handside of (\ref{boundW}) corresponding to the diagonal contribution dominates over the one corresponding to the non-diagonal contribution. Therefore, we can apply the same argument that we applied in Proposition \ref{prop1} to estimate the integral over this set. When $\alpha\in\grM'$ then it is the bound corresponding to the non-diagonal term the one which dominates. Let $I'_{M}$ be the contribution of $\grM'\setminus\grN$ to the integral $I_{M}$. By making use of Lemma \ref{lema1} and (\ref{Up}) we obtain that
$$I_{M}'\ll H^{t+t/24-\delta}M^{t}\int_{\grM'\setminus\grN}\lvert h(\alpha)\rvert^{s}\Upsilon(\alpha)^{t/2}d\alpha\ll H^{t+t/24-\delta}M^{t}(I_{1}+I_{2}),$$where 
$$I_{i}=\int_{\grM'\setminus\grN}\lvert h(\alpha)\rvert^{s-2}G_{i}(\alpha)\Upsilon(\alpha)^{t/2}d\alpha,\ \ \ \ \ \ \ \ \ \ \  \ \ \ \ \ \ i=1,2$$ with $G_{1}(\alpha)=\lvert h^{*}(\alpha)\rvert^{2}$ and $G_{2}(\alpha)=\lvert h(\alpha)-h^{*}(\alpha)\rvert^{2}.$ In view of the definitions (\ref{kioto}) and (\ref{kiotos}) for $\grN$ and $\grM'$ respectively, we make a distinction between the ranges $q\leq (6k)^{-1}H^{1/3}$ and $(6k)^{-1}H^{1/3}<q\leq M^{k}$. We also combine Lemmata \ref{expo} and \ref{lema4} with equations (\ref{Wint}) and (\ref{squa}) and the bound (\ref{tauk}) to obtain
\begin{align*}I_{1}\ll &
 P^{3s-3k}H^{-t/6-1/3}\sum_{q\leq (6k)^{-1}H^{1/3}}w_{k}(q)^{2}q^{t/2+1-t/2k+\varepsilon}
\\
&+P^{3s-3k}\sum_{(6k)^{-1}H^{1/3}<q\leq M^{k}}w_{k}(q)^{2}q^{1-t/2k+\varepsilon}\big(q^{-1}+P^{-1}\big)\ll P^{3s-3k+\varepsilon}H^{-t/6k}.\end{align*} Likewise, combining equations (\ref{Wint}) and (\ref{squa}) with Lemmata \ref{lema2} and \ref{expo} one finds that
\begin{align*}I_{2}\ll& P^{3s-3k-2+\varepsilon}H^{-t/6+2/3}\sum_{q\leq (6k)^{-1}H^{1/3}}q^{t/2-t/2k}w_{k}(q)^{2}
\\
&+P^{3s-3k-2+\varepsilon}\sum_{(6k)^{-1}H^{1/3}<q\leq M^{k}}w_{k}(q)^{2}q^{3-t/2k}\big(q^{-1}+P^{-1}\big)\ll P^{3s-3k+\varepsilon}H^{-t/6k},\end{align*}where we made use of (\ref{MH}). Therefore we obtain that $I'_{M}=O\big((HM)^{t}P^{3s-3k-\delta}\big),$ whence the result of the proposition follows combining (\ref{hWW}) with the previous estimates. 
\end{proof}
\emph{Proof of Theorem \ref{thm1.1} when $k=2,3$}. Note first that Lemma \ref{lemi} ensures positivity for $J_{\mathbf{Y},\mathbf{p}}(n)$ and guarantees that for $(\mathbf{Y},\mathbf{p})$ in the range described in the lemma then $J_{\mathbf{Y},\mathbf{p}}(n)\gg P^{s}H^{t/3}n^{-1}$. Similarly, Lemma \ref{lema6}  ensures the positivity of $\mathfrak{S}_{\mathbf{Y},\mathbf{p}}(n)$ and implies that for $(\mathbf{Y},\mathbf{p})$ satisfying the local conditions described after (\ref{Sin}) then $\mathfrak{S}_{\mathbf{Y},\mathbf{p}}(n)\gg 1$. As observed at the beginning of the lemmas, the intersection of the sets of pairs $(\mathbf{Y},\mathbf{p})$ satisfying those conditions has positive density. Therefore, we find that $$\sum_{\mathbf{Y},\mathbf{p}}\mathfrak{S}_{\mathbf{Y},\mathbf{p}}(n)J_{\mathbf{Y},\mathbf{p}}(n)\gg (HM)^{t}P^{3s-3k}(\log P)^{-t}.$$ Propositions \ref{prop1} and \ref{prop2} then yield the bound $R(n)\gg (HM)^{t}P^{3s-3k}(\log P)^{-t},$ which proves the theorem for $k=2,3$.

\section{The case $k=4$.}\label{sec7}
In this section we discuss the proof of the theorem for fourth powers. For such purpose, it is convenient to introduce the exponential sum
$$f(\alpha)=\sum_{x\in \mathcal{A}(P,P^{\eta})}e(x^{12}).$$
Let $R_{4}(n)$ be the number of solutions of the equation
$$n=\sum_{i=1}^{11}T(p_{i}\mathbf{x}_{i})^{4}+81\big(y_{1}^{12}+\dots+y_{46}^{12}\big),$$ where $\mathbf{x}_{i}\in \mathcal{W}$ with $M/2\leq p_{i}\leq M$ for $1\leq i\leq 11$ and $y_{i}\in\mathcal{A}(P,P^{\eta})$ for $1\leq i\leq 46$. Observe that the sums of three cubes on the right handside have been replaced by the specialization $3y^{3}$. Note as well that orthogonality yields the identity $$R_{4}(n)=\int_{0}^{1}W(\alpha)^{11}f(81\alpha)^{46}e(-\alpha n)d\alpha.$$
Our goal throughout the section is to obtain a lower bound for $R_{4}(n)$ for all sufficiently large $n$. Recalling (\ref{MH}) and (\ref{Wu}) and using the table of permissible exponents for $k=12$ in Vaughan and Wooley \cite{VauWoo2} we find that
\begin{align}\label{wf}\int_{\grm}\lvert W(\alpha)\rvert^{11}\lvert f(81\alpha)\rvert^{46}d\alpha&
 \ll H^{11+11/24-\delta}M^{11/2}\int_{0}^{1}\lvert f(\alpha)\rvert^{46}d\alpha\nonumber
\\
&\ll (HM)^{11}P^{34+\Delta_{23}-1/2-\delta},\end{align}where $\Delta_{23}=0.4988383,$ and hence it follows that the minor arc contribution is then $O\big((HM)^{11}P^{34-\delta}\big).$ 

We define a set of narrow major arcs $\grP$ by taking the union of
\begin{equation*}\grP(a,q)=\Big\{ \alpha\in [0,1): \Big\lvert \alpha-a/q\Big\rvert \leq \frac{R}{n}\Big\}\end{equation*} with $0\leq a\leq q\leq R$ and $(a,q)=1$, where $R=(\log P)^{1/5}$, and consider $\mathfrak{p}=[0,1)\setminus \grP$. In the next few lines we will combine all sort of major and minor arc techniques to prune back to the set of narrow arcs $\grP.$ As observed after (\ref{kiotos}), whenever $\alpha\in \grM\setminus\grM'$ then the bound in the right handside of (\ref{boundW}) corresponding to the diagonal contribution dominates over the one corresponding to the non-diagonal contribution. Therefore, we can apply the same argument that we applied in (\ref{wf}) to obtain that the integral over that set is $O\big((HM)^{11}P^{34-\delta}\big).$ 

We next note for further purposes that Theorem 1.8 of Vaughan \cite{Vau3} yields
\begin{equation}\label{f81}\sup_{\grn}\lvert f(81\alpha)\rvert\ll P^{1-\rho+\varepsilon},\end{equation} where $\rho=0.004259.$ As experts will realise, one could obtain a slightly bigger $\rho$ by applying the methods in \cite{Woo}. For the sake of brevity though, we avoid that treatment and make use of the weaker version of the estimate. We also remark that such improvement in the exponent would make no impact in the argument. Observe that using the same procedure as in (\ref{Wsa}) and (\ref{Wint}) we deduce that
\begin{equation}\label{fev}\sum_{a=1}^{q}\lvert f\big(81(\beta+a/q)\big)\rvert^{12}\ll q^{1+\varepsilon}P^{12}\max(q^{-1}+P^{-1}).\end{equation}
Note as well that whenever $\alpha\in \grM'\setminus \grN$ then $(1+n\lvert\beta\rvert)^{3/2}\geq H^{1/3}q^{-1}$, and hence Lemmata \ref{lema3} and \ref{lema4} yield $W(\alpha)\ll MH^{2/3}q^{1+\varepsilon}w_{4}(q)(1+n\lvert\beta\rvert)^{1/2}.$ By the preceding discussion together with Lemma \ref{lema1} and equations (\ref{f81}) and (\ref{fev}) we obtain
\begin{align*}\int_{\grM'/\grN}\lvert W(\alpha)\rvert^{11}\lvert f(81\alpha)\rvert^{46}d\alpha\ll (HM)^{11}P^{34(1-\rho)}\sum_{q\leq M^{4}}q^{2}\tau_{4}(q)^{4}w_{4}(q)(q^{-1}+P^{-1}).
\end{align*}

Here the reader may find useful to observe that we applied the estimate (\ref{boundW}) to eight copies of $W(\alpha)$ and the bound for $W(\alpha)$ deduced above to just one of them. Likewise, we made use of the pointwise estimate (\ref{f81}) to bound $34$ copies of $f(81\alpha)$ and we used the other $12$ to exploit the congruence condition via (\ref{fev}). We get that the above sum when $q\leq P$ is $O\big((HM)^{11}P^{34-\delta}\big)$ via Lemma \ref{expo}. Similarly, we use Lemma \ref{expo} and the bound $qP^{-1}\leq P^{1/11}$, which follows after (\ref{MH}), for the range $P\leq q\leq M^{4}$ to obtain that such contribution is also $O\big((HM)^{11}P^{34-\delta}\big)$. By the observation made before (\ref{ech}), which is still valid for $k=4$, and Lemma \ref{lema3} we find that whenever $\alpha\in\grN$ then
$$W(\alpha)\ll \frac{q^{\varepsilon}w_{4}(q)HM}{(1+n\lvert\beta\rvert)(\log P)}.$$ Therefore, the application of this bound and (\ref{Wsa}) yield
\begin{align*}\int_{\grN/\grP}\lvert W(\alpha)\rvert^{11}\lvert f(81\alpha)\rvert^{46}d\alpha\ll&
 (HM)^{11}P^{34}(\log P)^{-11}R^{-6}\sum_{q\leq R}q^{\varepsilon}w_{4}(q)^{7}
\\
&+(HM)^{11}P^{34}(\log P)^{-11}\sum_{q>R}q^{\varepsilon}w_{4}(q)^{7}.
\end{align*}Consequently, Lemma \ref{expo} and (\ref{wuok}) imply that such integral is $O\big((HM)^{11}P^{34}(\log P)^{-11-\delta}\big)$.

In what follows, we will briefly describe the singular series associated to the problem. There might be other approaches that would lead to more precise asymptotic formulae, but for the sake of simplicity we avoid including the sums of three cubes in the singular series. On recalling (\ref{Saq}), it is convenient to consider, for an integer $m\in\mathbb{N}$ and a prime $p$ the sums
$$S_{m}(q)=q^{-46}\sum_{\substack{a=1\\ (a,q)=1}}^{q}S_{12}(q,81a)^{46}e_{q}\big(-a(n-m)\big), \ \ \ \ \ \ \ \ \ \ \ \sigma_{m}(p)=\sum_{h=0}^{\infty}S_{m}(p^{h}).$$ Observe that whenever $3\nmid q$ then we can make a change of variables to rewrite $S_{m}(q)$ as 
$$S_{m}(q)=q^{-46}\sum_{\substack{a=1\\ (a,q)=1}}^{q}S_{12}(q,a)^{46}e_{q}\big(-a\overline{81}^{-1}(n-m)\big),$$where $\overline{81}^{-1}$ denotes the inverse of $81\mmod{q}$. Note as well that Lemma 3 of \cite{Vau1} yields the bound $S_{m}(q)\ll q\tau_{12}(q)^{46}$, which implies that $\sigma_{m}(p)=1+O(p^{-22})$ and delivers the convergence of the singular series
\begin{equation}\label{SM}\frak{S}_{m}(n)=\sum_{q=1}^{\infty}S_{m}(q)\end{equation} and its upper bound $\frak{S}_{m}(n)\ll 1$. Here the reader may find useful to observe that we implicitly used the multiplicativity of $S_{m}(q)$ and the expression of the singular series as the product $$\frak{S}_{m}(n)=\prod_{p}\sigma_{m}(p).$$The estimate for $S_{m}(q)$ mentioned before (\ref{SM}) also delivers, for $Q\geq 1$, the bound
\begin{equation}\label{ssm}\sum_{q>Q}\lvert S_{m}(q)\rvert\ll Q^{-\alpha}\end{equation} for some $\alpha>0$ via the ideas employed in the proof of Lemma \ref{lema5}. Observe that by Lemmata 2.12, 2.13 and 2.15 of \cite{Vau} one gets for every prime $p\neq 3$ the lower bound $\sigma_{m}(p)\geq p^{-45\gamma}$, where $\gamma=3$ when $p=2$ and $\gamma=1$ otherwise. Likewise, note that when $m\equiv n\mmod{81}$ and $h\geq 5$ orthogonality yields
$$\sum_{l=0}^{h}S_{m}(3^{l})=3^{-45h}M_{n,m}(3^{h}),$$where $M_{n,m}(3^{h})$ denotes the number of solutions of the congruence
$$x_{1}^{12}+\dots+x_{46}^{12}\equiv (n-m)/81\mmod{3^{h-4}}$$with $1\leq x_{i}\leq 3^{h}.$ Therefore, the application of Lemmata 2.13 and 2.15 of \cite{Vau} gives $\sigma_{m}(3)\geq 3^{-86}.$ Consequently, combining these lower bounds with the fact that $\sigma_{m}(p)-1=O(p^{-22})$ we obtain $\frak{S}_{m}(n)\gg 1.$ Observe as well that the preceding discussion yields $\frak{S}_{m}(n)\geq 0$ for every $m\in\mathbb{N}.$

 Before showing a lower bound of the expected size for the contribution of the set of narrow arcs, we introduce for convenience the weighted exponential sum
$$w(\beta)=\sum_{P^{12\eta}<x\leq n}\frac{1}{12}x^{-11/12}\rho\Big(\frac{\log x}{12\eta\log P}\Big)e(\beta x),$$
where $\rho$ denotes the Dickman's function, defined for real $x$ by
$$\rho(x)=0\text{ when } x<0,$$
$$\rho(x)=1 \text{ when } 0\leq x\leq 1,$$
$$\rho \text{ continuous for } x>0,$$
$$\rho \text{ differentiable for } x>1$$
$$x\rho'(x)=-\rho(x-1) \text{ when } x>1.$$For the sake of simplicity, we define the auxiliary function $f^{*}(\alpha)$ by putting $f^{*}(\alpha)=q^{-1}S(q,81a)w\big(81(\alpha-a/q)\big)$ when $\alpha\in\grP(a,q)\subset\grP$ and $f^{*}(\alpha)=0$ for $\alpha\in\grp.$ Then, it is a consequence of Vaughan\footnote{Observe that the condition $(a,q)=1$ in Vaughan \cite[Lemma 5.4]{Vau3} can be relaxed to $(a,q)=C$ for some constant $C$.} \cite[Lemma 5.4]{Vau3} that for $\alpha\in\grP(a,q)\subset\grP$ one has $f(81\alpha)-f^{*}(\alpha)=O\big(PR^{-3}\big)$ and $f^{*}(\alpha)\ll q^{-1/12}P(1+n\lvert\beta\rvert)^{-1/12}.$ Moreover, by the methods of Vaughan \cite[Lemma 2.8]{Vau} and the monotonicity of $\rho$ it follows that
\begin{equation}\label{www}w(\beta)\ll \frac{P}{(1+n\lvert\beta\rvert)^{1/12}}.\end{equation} Finally, when $m\in\mathbb{N}$ it is convenient to introduce $K(m)$, defined as the number of solutions of the equation
$$m=T(p_{1}\mathbf{x}_{1})^{4}+\dots+T(p_{11}\mathbf{x}_{11})^{4}$$ for $\mathbf{x}_{i}\in\mathcal{W}$ and $M/2\leq p_{i}\leq M.$ Combining the estimates mentioned before (\ref{www}) we obtain that
\begin{align*}\int_{\grP}W(\alpha)^{11}f(81\alpha)^{46}e(-\alpha n)d\alpha=&
\sum_{m\leq 11n}K(m)\int_{\grP}f^{*}(\alpha)^{46}e\big(-\alpha(n-m)\big)d\alpha
\\
&+O\big((HM)^{11}P^{34}(\log P)^{-11-\delta}\big).
\end{align*}Observe that the main term on the right can be written as
\begin{equation}\label{main}\sum_{m\leq 11n}K(m)\sum_{q\leq R}S_{m}(q)\int_{\lvert\beta\rvert\leq n^{-1}R}w(81\beta)^{46}e\big(-\beta (n-m)\big)d\beta.\end{equation} By (\ref{www}) we obtain that the integral on the above expression over the range $\lvert\beta\rvert>n^{-1}R$ is $O(P^{34}(\log P)^{-\delta})$. Therefore, an application of this observation and (\ref{ssm}) gives that the contribution of the set of narrow arcs $\grP$ is
$$\sum_{m\leq 11n}K(m)\frak{S}_{m}(n)\int_{0}^{1}w(81\beta)^{46}e\big(-\beta(n-m)\big)d\beta+O\big((HM)^{11}P^{34}(\log P)^{-11-\delta}\big).$$ We further note that whenever $P^{12\eta}< x\leq n$ then $$\rho\Big(\frac{\log x}{12\eta\log P}\Big)\gg 1,$$ so combining the positivity of $\frak{S}_{m}(n)$, orthogonality and the lower bound $\frak{S}_{m}(n)\gg 1$ when $m\equiv n\mmod{81}$ we obtain that (\ref{main}) is bounded below by $$\sum_{\substack{m\leq 11n/12\\ m\equiv n\mmod{81}}}K(m)(n-m)^{17/6}.$$One can check via an application of Hensel's Lemma\footnote{Here the reader may find useful to observe that the set of sums of three cubes modulo $27$ are the residue classes not congruent to $4$ or $5$ modulo $9$.} and Lemma 2.14 of \cite{Vau} that the set of numbers of the shape $T(p_{1}\mathbf{x}_{1})^{4}+\dots+T(p_{11}\mathbf{x}_{11})^{4}$ with $p_{i},\mathbf{x}_{i}\leq 81$ covers all the residue classes modulo $81$. Consequently, by the preceding discussion we find that
$$\int_{\grP}W(\alpha)^{11}f(81\alpha)^{46}e(-\alpha n)d\alpha\gg (HM)^{11}P^{34}(\log P)^{-11},$$ which combined with the estimates obtained through the pruning process yields $R_{4}(n)\gg (HM)^{11}P^{34}(\log P)^{-11}.$

\end{document}